\documentclass[11pt]{amsart}
\usepackage{amsfonts}
\usepackage{amsmath}
\usepackage{amssymb}
\usepackage{amsthm}
\usepackage{enumerate}
\usepackage{mathdots}
\usepackage[hidelinks]{hyperref}
\usepackage{caption}
\usepackage{adjustbox}
\usepackage{bbold}
\usepackage{mathtools}
\usepackage{hyperref}
\newtheorem{theorem}{Theorem}
\newtheorem{lemma}{Lemma}
\newtheorem{proposition}{Proposition}
\newtheorem{definition}{Definition}

\newtheorem{cor}{Corollary}

\newtheorem*{rmk*}{Remark}

\DeclareMathOperator\supp{supp}
\usepackage{fancyhdr}

\title{On unique sums in Abelian groups}
\author{Benjamin Bedert\\
\tiny Mathematical Institute, University of Oxford}
\thanks{benjamin.bedert@magd.ox.ac.uk\\The author gratefully acknowledges financial support from the EPSRC}
\begin{document}
\maketitle
\begin{abstract}
Let $A$ be a subset of the cyclic group $\mathbf{Z}/p\mathbf{Z}$ with $p$ prime. It is a well-studied problem to determine how small $|A|$ can be if there is no unique sum in $A+A$, meaning that for every two elements $a_1,a_2\in A$, there exist $a_1',a_2'\in A$ such that $a_1+a_2=a_1'+a_2'$ and $\{a_1,a_2\}\neq \{a_1',a_2'\}$. Let $m(p)$ be the size of a smallest subset of $\mathbf{Z}/p\mathbf{Z}$ with no unique sum. The previous best known bounds are $\log p \ll m(p)\ll \sqrt{p}$. In this paper we improve both the upper and lower bounds to $\omega(p)\log p \leqslant m(p)\ll (\log p)^2$ for some function $\omega(p)$ which tends to infinity as $p\to \infty$. In particular, this shows that for any $B\subset \mathbf{Z}/p\mathbf{Z}$ of size $|B|<\omega(p)\log p$, its sumset $B+B$ contains a unique sum. We also obtain corresponding bounds on the size of the smallest subset of a general Abelian group having no unique sum. 
\end{abstract}

\tableofcontents
\section{Introduction}
Let $A,B$ be subsets of an finite Abelian group $G$. Their sumset $A+B$ is defined as $A+B = \{a+b:a\in A, b\in B\}$. We say that $A$ has a unique sum if there exist $a_1,a_2$ in $A$ so that the only solutions to $x+y=a_1+a_2$ with $x,y\in A$ are the trivial ones $(x,y)=(a_1,a_2),(a_2,a_1)$. In this case, we say that $a_1+a_2$ is a unique sum in $A+A$. In this paper, we will study the conditions under which a set $A$ must contain a unique sum. In particular, given any finite Abelian group $G$, we want to determine the size of the smallest subset of $G$ having no unique sum.
\begin{definition}
\normalfont
Let $G$ be a finite Abelian group. Then we define $m(G)$ to be the size of the smallest subset of $G$ which has no unique sum. Equivalently, $m(G)$ is the smallest integer so that any subset $B\subset G$ with size $|B|<m(G)$ has a unique sum.
Of special importance is the case where $G=\mathbf{Z}/p\mathbf{Z}$ is the cyclic group of prime order $p$ so we abbreviate the notation and write $m(p)$ for $m\left(\mathbf{Z}/p\mathbf{Z}\right)$.
\end{definition}
The question of estimating $m(p)$ was explicitly asked by S. Kopparty (open problems session, Harvard 2017) and it also appears as Problem 27 on B. Green's list of 100 open problems \cite{green}. Questions of this type go back at least to a paper of Straus \cite{straus} in which he proved the first bounds on the size $f(p)$ of the smallest subset $A\subset \mathbf{Z}/p\mathbf{Z}$ having no unique difference. Here, we say that $A$ contains a unique difference if there exist $a_1,a_2\in A$ such that the only solution to $x-y=a_1-a_2$ with $x,y\in A$ is the trivial one $(x,y)=(a_1,a_2)$. Straus proved that $f(p) \geqslant 1+\log_4(p-1)$ and this was later improved by Browkin, Divi\v{s} and Schinzel \cite{browkin} who obtained the following.
\begin{theorem}[Browkin-Divi\v{s}-Schinzel, \cite{browkin}]
Let $p$ be prime and $A,B\subset\mathbf{Z}/p\mathbf{Z}$.
\begin{itemize}
    \item[(i)] If $p>\min\left(2^{|A|+|B|-2},|A|^{|B|-1},|B|^{|A|-1}\right)$, then $A+B$ contains a unique sum.
    \item[(ii)] If $p>2^{|A|-1}$, then A has a unique difference and a unique sum.
\end{itemize}
\end{theorem}
Their result was extended to general Abelian groups by Lev.
\begin{theorem}[Lev, \cite{levrect}]
Let $A, B$ be subsets of a finite abelian group $G$ and let $p(G)$ be the smallest prime
divisor of $|G|$.
\begin{itemize}
    \item[(i)] If $p(G) > 2^{|A|+|B|-3}$, then $A+B$ contains a unique sum.
    \item[(ii)] If $p(G) > 2^{|A|-1}$, then $A$ has a unique difference and a unique sum.
\end{itemize}
\label{levbound}
\end{theorem}
The current best bound for unique sums in $A+B$ is due to Leung and Schmidt \cite{schmidt}, who recently proved under the same assumptions as in Theorem \ref{levbound} that $A+B$ contains a unique sum if $p(G) > (\sqrt[4]{12})^{|A|+|B|-2}$. Closely related problems, such as estimating the size of the smallest $A\subset \mathbf{Z}/p\mathbf{Z}$ so that any sum in $A+A$ has at least $K$ distinct representations, or alternatively such that $\underbrace{A+\dots+A}_{k}$ has no unique sum, have also been studied, for a selection see \cite{  DucX, hill, lucz, Kreps}.
\bigskip

These bounds show that the size $f(G)$ of the smallest subset of $G$ with no unique difference satisfies $f(G)\gg\log p(G)$, and examples of sets with no unique difference and size $O(\log p(G))$ do exist and already appear in Straus's original paper \cite{straus}. Hence, we have $f(G)=\Theta(\log p(G))$. For the problem of determining the size $m(G)$ of the smallest $A\subset G$ having no unique sum, the results above provide a lower bound of the shape $m(G)\geqslant C\log p(G)$ for some absolute constant $C>0$, which is the current record lower bound. Unlike the situation for sets with no unique difference, there are no constructions known of sets with no unique sum and size $O(\log p(G))$.
The following two theorems are our main results proving that such examples cannot exist and they are the first lower bounds on $m(G)$ replacing the constant $C$ in the bound above by a function tending to infinity with $p$.
\begin{theorem}
    There is a function $\omega(n)$ which tends to infinity as $n\to\infty$ such that the following holds. Let $p$ be a prime, then $m(p)\geqslant \omega(p)\log p$. In fact, one can take
    $$\omega(n)\gg\frac{\sqrt{\log\log\log n}}{\log\log\log\log n}.$$ In particular, if $B\subset \mathbf{Z}/p\mathbf{Z}$ has size $|B|<\omega(p)\log p$, then $B$ has a unique sum.
\label{lowerboundthm}
\end{theorem}
Our goal is to obtain a lower bound with $\omega(n)\to \infty$ and we have not tried to optimise the exact shape of $\omega$, which can certainly be improved. To analyse $m(G)$ for general Abelian groups, we begin with the simple observation that if $p$ is a prime dividing the order of $G$, then $G$ contains a cyclic group of order $p$ as a subgroup. Thus for a general Abelian group $G$ we have
$$m(G)\leqslant \min_{p\text{ prime, }p| |G|} m(p).$$
Hence, there is no hope of proving a better lower bound on the size of a subset of $G$ having no unique sum than those holding in cyclic groups $\mathbf{Z}/p\mathbf{Z}$ with $p| |G|$. The following theorem shows that we can get a lower bound of the form of Theorem \ref{lowerboundthm} in general.
\begin{theorem}
    Let $G$ be a finite Abelian group and let $p(G)$ be the smallest prime factor of $|G|$. If $A\subset G$ has no unique sum, then
    $$|A|\geqslant \omega(p(G))\log p(G),$$
    where $\omega$ is the same function as in Theorem \ref{lowerboundthm}.
\label{generallowerboundthm}
\end{theorem}
We also improve on the best known upper bound on $m(p)$ by constructing for each prime $p$ a set $A$ which has no unique sum and size $O((\log p)^2)$. This improves the previous best known bound $m(p)\ll \sqrt{p}$ which came from a rather easy construction of a set $A\subset \mathbf{Z}/p\mathbf{Z}$ whose sumset $A+A$ is the whole of $\mathbf{Z}/p\mathbf{Z}$.
\begin{theorem}
    Let $p$ be a prime, then $m(p)\ll (\log p)^2$. That is, for every prime $p$ there is a set $A$ of size $O((\log p)^2)$ having no unique sum.
\label{upperboundthm}
\end{theorem}
It is clear that this implies the corresponding bound $m(G)\ll (\log p(G))^2$ for general Abelian groups $G$.

\bigskip

\textbf{Acknowledgements.}
The author would like to thank his supervisor Ben Green for introducing him to the problem, and Swastik Kopparty for sharing helpful references. The author also gratefully acknowledges financial support from the EPSRC.

\section{Prerequisites}
In this paper, $G$ denotes an Abelian group and we shall always write $+$ for the group operation. We write $p(G)$ for the smallest prime factor of $|G|$. To improve readability, we omit floor and ceiling functions throughout the paper, but it will be clear from context which quantities should be integer-valued. For an element $g\in G$, $r_A(g)$ denotes the number of ordered pairs in $A^2$ whose sum equals $g$, so
\begin{equation*}
    r_A(g) \vcentcolon=\left|\{(a,a')\in A^2: a+a'=g\}\right|.
\end{equation*}
So a set $A\subset G$ has a unique sum if and only if there is some $g$ such that $1\leqslant r_A(g) \leqslant 2$. We introduce the notion of Freiman-isomorphic sets, which will play a crucial role in our argument.
\begin{definition}
    Let $G,G'$ be Abelian groups and let $A\subset G$, $A'\subset G'$. We say that a map $\phi:A\to A'$ is a \textit{Freiman homomorphism} if whenever $a_1,a_2,a_3,a_4\in A$ satisfy $$a_1+a_2=a_3+a_4,$$
    then $$\phi(a_1)+\phi(a_2)=\phi(a_3)+\phi(a_4).$$
    We say that $A$ and $A'$ are \textit{Freiman-isomorphic} if there is a bijective Freiman homomorphism $\phi:A\to A'$ so that $\phi^{-1}$ is also a Freiman homomorphism.
\label{frei}
\end{definition}
We continue with two useful lemmas which are part of a large number of results in the literature that are often referred to as `rectification' results. Such results show that under certain assumptions on a subset $A$ of an Abelian group $G$, $A$ is Freiman-isomorphic to a set of integers. If this is the case, we say that $A$ is rectifiable. The rectification principle that we need states that small subsets of Abelian groups are rectifiable, as made precise in the following two lemmas.
\begin{lemma}[Bilu-Lev-Ruzsa, \cite{bilu} Theorem 3.1]
Let $p$ be prime and let $Z\subset \mathbf{Z}/p\mathbf{Z}$ have size $|Z|\leqslant \log_2 p$. Then $Z$ is Freiman-isomorphic to a set of integers.
\label{rectificationp}
\end{lemma}
We will use the following generalisation to arbitrary Abelian groups.
\begin{lemma}[Lev, \cite{levrect} Theorem 1]
Let $G$ be a finite Abelian group and let $p(G)$ denote the smallest prime dividing $|G|$. If $Z\subset G$ has size $|Z|\leqslant \log_2 p(G)$, then $Z$ is Freiman-isomorphic to a set of integers.
\label{rectificationG}
\end{lemma}
We show how one can easily recover the previous best known lower bound $m(G)\gg \log p(G)$ using these lemmas. Indeed, suppose that $A\subset G$ has no unique sum. Then by definition, neither does any set that is Freiman-isomorphic to $A$. In particular, $A$ cannot be rectifiable since any finite set of integers $A'$ trivially has a unique sum, namely $\max A'+\max A'$. Thus, Lemma \ref{rectificationG} implies that $|A|> \log_2 p(G)$ as desired. Such arguments using rectification to find a unique sum go back to Straus's original paper \cite{straus}, and note that one can deduce Theorem \ref{levbound}  from Lemma \ref{rectificationG} in this way.
\bigskip

\textbf{Outline of the proof.} In section 3, we prove a general structural result about sets $Z\subset G$ with large additive span, by which we mean that $\Sigma(Z)=\left\{\sum_{z'\in Z'} z': Z'\subset Z\right\}$ is large. To be precise, we show that if $\left|\Sigma(Z)\right|$ is large, then $Z$ contains a large dissociated subset. In the additive combinatorics literature, the size of the largest dissociated subset of $Z$ is often referred to as the additive dimension $\dim(Z)$ of $Z$. Using this terminology, we show precisely in section 3 that sets with large additive span have large additive dimension. In section 4, we show that if $A\subset G$ is a set having no unique sum, then some translate $A+g$ of $A$ has very large additive span in the sense that $\Sigma(A+g)$ contains a non-trivial subgroup of $G$. For the most interesting case where $G=\mathbf{Z}/p\mathbf{Z}$, this shows that $\Sigma(A+g)=\mathbf{Z}/p\mathbf{Z}$ is the whole group.
\medskip

Combining the results from sections 3 and 4, we obtain that if $A$ has no unique sum, then some translate  $A+g$ has large additive dimension. Note that $A'=A+g$ also contains no unique sum. Finally, in section 5 we employ a density increment argument to prove that a set $A'$ with no unique sum cannot contain a dense dissociated subset, i.e. we show that $\dim(A')=o_{p(G)}(1) \cdot |A'|$ as $p(G)\to \infty$. As sections 3 and 4 imply that $\dim(A')$ is large, this will yield the required lower bound $|A|=|A'|\geqslant \omega(p(G)) \log p(G)$.

\section{Sets with small dimension have small additive span}
In this short section, we will prove an inequality that holds for any subset $Z$ of an Abelian group $G$. The proof of this result is self-contained and one can forget about sets having no unique sum in this whole section. We begin with three important definitions from additive combinatorics.
\begin{definition}
    \normalfont Let $G$ be a finite Abelian group and let $S\subset G$. We say that $S$ is \textit{dissociated} if whenever there exist $(\mu_s)_{s\in S}\in \{-1,0,1\}^{S}$ so that $\sum_{s\in S}\mu_s s = 0$, then $\mu_s = 0$ for all $s\in S$. Equivalently, $S$ is dissociated if whenever $S_1,S_2\subset S$ with $\sum_{s\in S_1}s=\sum_{s\in S_2}s$, then $S_1=S_2$.
\end{definition} 
The notion of additive dimension is an important concept in additive combinatorics and there is an extensive literature on this topic, see for example \cite{schoen,ilyaserg} and the references therein.
\begin{definition}
    \normalfont Let $G$ be a finite Abelian group and let $S\subset G$. Then we define the \textit{additive dimension} $\dim(S)$ of $S$ to be the size of the largest dissociated subset of $S$.
\end{definition}
The use of the word dimension in this setting is natural in light of the following observation.
\begin{lemma}
    If $S\subset G$ and $D\subset S$ is a maximal dissociated subset of size $|D|=\dim(S)$, then $S$ is contained in the cube $$\left\{\sum_{d\in D} \mu_d d: \mu_d\in \{-1,0,1\}\right\}.$$
\label{dimlem}
\end{lemma}
\begin{proof}
Let $s\in S$, if $s\in D$ then $s$ trivially lies in this additive cube. Otherwise, $D\cup\{s\}$ is a strictly larger subset of $S$ so not dissociated whence we get a non-trivial relation of the form $\mu_ss+\sum_{d\in D}\mu_dd=0$. As $D$ is dissociated, $\mu_s\neq 0$ and the result follows.
\end{proof}
We need one more definition.
\begin{definition}
\normalfont
    Let $G$ be a finite Abelian group. For a subset $Z\subset G$ we define its \textit{additive span} to be the set 
    \begin{equation}\Sigma (Z)\vcentcolon = \left\{\sum_{z\in Z} \varepsilon_z z: \varepsilon_z \in \{0,1\}\right\} = \left\{\sum_{z\in Z'} z:Z'\subseteq Z\right\}.
    \label{addispandefi}
    \end{equation}
This definition also makes sense when $Z$ is a finite multiset consisting of elements of $G$. In this case, every $z\in Z$ appears $k$ times in the sum $\sum_{z \in Z}\varepsilon_zz$ in \eqref{addispandefi} if $Z$ contains $k$ copies of $z$.
We say that a (multi-)set $S\subset G$ is an \textit{additive basis} for $G$ if $\Sigma(S) = G$. In other words, $S$ is an additive basis if for every element $g\in G$, there is some (multi-)subset $S_g \subseteq S$ whose elements sum to $g$.
\end{definition}
Our aim in this section is to find an upper bound on $\left|\Sigma(Z)\right|$. Observe that the trivial bound $\left|\Sigma(Z)\right|\leqslant 2^{|Z|}$ always holds. In general, one can of course not improve on this trivial bound as $\left|\Sigma(Z)\right|=2^{|Z|}$ if $Z$ is a dissociated set. Similarly, the additive span $\Sigma(Z)$ will be large if $Z$ contains a fairly large subset which is dissociated. It is therefore natural to wonder if this is in a sense the only reason why $\Sigma(Z)$ can be large, meaning that if $\Sigma(Z)$ is large then it implies that $Z$ contains a large dissociated subset. The following proposition shows that this result is indeed true.
\begin{proposition}
    Let $G$ be an Abelian group and let $Z$ be a finite multiset consisting of elements of $G$. Then
    \begin{equation}
    \left|\Sigma(Z)\right|\leqslant {|Z|\choose \dim(Z)}{|Z|+\dim(Z)\choose \dim(Z)}.
    \label{addispanineq}
    \end{equation}
\label{addispanineqprop}
\end{proposition}
Hence, if $\left|\Sigma(Z)\right|$ is large, then $\dim(Z)$ is large which means precisely that $Z$ has a large dissociated subset. We state a bound which is more useful in practice.
\begin{cor}
    Let $G$ be an Abelian group and let $Z$ be a finite multiset consisting of elements of $G$. Then
    \begin{equation}
    \left|\Sigma(Z)\right|\leqslant 2^{2 \dim(Z)\cdot\Big(\log_2\left(\frac{|Z|}{\dim(Z)}\right)+2\Big)} = \left(\frac{4|Z|}{\dim(Z)}\right)^{2 \dim(Z)}.
    \label{addispanboundcor}
    \end{equation}
\label{addispanineqcor}
    
\end{cor}
Clearly, we always have the lower bound $\left|\Sigma(Z)\right|\geqslant 2^{\dim(Z)}$ and one may ask if the extra factor $\log_2\left(\frac{|Z|}{\dim(Z)}\right)$ in the exponent in \eqref{addispanboundcor} is necessary. The following example shows that in fact it is necessary. Pick an integer $d$ and consider $G=\mathbf{Z}^d$ with standard generating set $\{e_1,e_2,\dots,e_d\}$. Then we can take $Z$ to be the the multiset consisting of $k$ copies of each $e_i$ with $1\leqslant i\leqslant d$. It is easy to see that $\dim(Z)=d$, but $\Sigma(Z) = \left\{\sum_{i=1}^d n_ie_i: 0\leqslant n_i\leqslant k\text{ for each $i$}\right\}$ has size $(k+1)^d > \left(\frac{|Z|}{d}\right)^{d}$. Let us now give the proof of Proposition \ref{addispanineqprop}.
\begin{proof}[Proof of Proposition \ref{addispanineqprop}]
We consider an element $y\in \Sigma(Z)$, so we may find coefficients $\varepsilon_z(y) \in \{0,1\}$ for $z\in Z$ so that \begin{equation}
    y=\sum_{z\in Z} \varepsilon_z(y) z,
    \label{yrepr}
\end{equation}
where each $z$ in the multiset $Z$ occurs with multiplicity in this sum. Our idea is to use a type of compression on these sums until every $y$ can be expressed as a sum of elements in $Z$ with small support. We will use a different type of compression in a later section, so to avoid confusion we call the type of compressions used in this section `support-compressions'. For an expression $y=\sum_{z\in Z} n_z z$ with non-negative integers $n_z$ which is not already maximally compressed, a `support-compression' yields a new expression $y=\sum_z m_z z$ with smaller support, i.e. the multiset $\{z\in Z: m_z\neq 0\}$ is smaller than $\{z\in Z: n_z\neq 0\}$. Repeatedly applying this shows that every $y$ in $\Sigma(Z)$ can be expressed as a sum of elements of $Z$ of the form $y=\sum_{z\in Z} n_z z$ whose support is a dissociated subset of $Z$ and with $\sum_z n_z$ not too large. A combinatorial counting argument then yields \eqref{addispanineq}.
\smallskip

We make this argument precise. For each $y\in \Sigma(Z)$, define \begin{equation}
    a(y)\vcentcolon=\sum_{z\in Z}\varepsilon_z(y) \in \{0,1,\dots, |Z|\}
    \label{a(y)defi}
\end{equation} where the $\varepsilon_z(y)\in \{0,1\}$ are so that \eqref{yrepr} holds (if there is more than one choice, we just pick one of these arbitrarily). We now define the following set
\begin{align}
    T(y)&\vcentcolon= \Big\{(n_z)_{z\in Z}\in \mathbf{N}^{Z}: y = \sum_{z\in Z} n_z z\text{ and }\sum_{z\in Z} n_z\leqslant a(y)\Big\}.
\end{align}
For each $|Z|-$tuple $(n_z)$ in $T(y)$, we define its support-size as follows
\begin{equation}
    \supp((n_z)_{z\in Z}) \vcentcolon= \left|\{z\in Z: n_z\neq 0\}\right|
\end{equation}
counted with multiplicity if $Z$ is a multiset. We begin by noting that the set $T(y)$ is non-empty because $y=\sum_{z\in Z}\varepsilon_z(y) z$ by \eqref{yrepr} and $\sum_z \varepsilon_z(y) = a(y)$ by \eqref{a(y)defi} so $(\varepsilon_z(y))_{z\in Z}\in T(y)$. Hence, we can consider an element $(k_z(y))_{z\in Z}\in T(y)$ with minimal support-size $\supp((k_z(y))_{z\in Z})$. Let $K=K(y)=\{z\in Z: k_z(y)\neq 0\}$ be its support, so $K$ is a multisubset of $Z$ and we obtain the following information about $K$.
\begin{lemma}
    Let $y\in \Sigma(Z)$ and let $(k_z)_{z\in Z}\in T(y)$ be chosen to minimise $\supp((k_z)_{z\in Z})$ over all elements of $T(y)$. Then $K=\{z\in Z: k_z\neq 0\}$ is a dissociated subset of $Z$.
\label{dissK}
\end{lemma}
\begin{proof}
Suppose for a contradiction that the multiset $K$ is not dissociated. Hence there exist distinct multisubsets $K_1,K_2$ of $K$ so that $\sum_{z\in K_1}z=\sum_{z\in K_2}z$. We may further assume that $K_1$ and $K_2$ are disjoint as removing common elements in $K_1\cap K_2$ from both multisets does not change that both multisets have equal sum. Also assume that $|K_1|\geqslant |K_2|$ and define $k^- = \min_{z\in K_1} k_z$.
    Then we can write
\begin{align}
    y &= \sum_{z\in K} k_z z \nonumber\\
    &= \sum_{z\in K\setminus{(K_1\cup K_2)}} k_z z +\sum_{z\in K_1} (k_z-k^-)z+\sum_{z\in K_2} (k_z+k^-)z \label{rewriteiterate}    
\end{align}
so we can construct a new tuple $(k'_z)_{z\in Z}$ as follows by defining:
\begin{equation}
       k'_z= \left\{
  \begin{array}{@{}ll@{}}
    k_z, & \text{if}\ z\in Z\setminus{(K_1\cup K_2)} \\
    k_z-k^-, & \text{if}\ z\in K_1 \\
    k_z+k^-, & \text{if}\ z\in K_2. 
  \end{array}\right.
\label{k'definition}
    \end{equation}
We proceed by showing that $(k'_z)_{z\in Z}\in T(y)$. First, it is clear from the definition \eqref{k'definition} that each $k'_z$ is a non-negative integer as $(k_z)_{z\in Z}\in T(y)$ and $k^-\leqslant k_z$ for all $z\in K_1$. From \eqref{rewriteiterate}, we see that
$$y = \sum_{z\in Z} k'_z z.$$ Finally, from \eqref{k'definition} we observe that 
\begin{align*}
    \sum_{z\in Z}k'_z &= \sum_{z\in Z} k_z -k^-|K_1|+k^-|K_2|\\
    &\leqslant \sum_{z\in Z}k_z\\
    &\leqslant a(y)
\end{align*}
using that $|K_1|\geqslant |K_2|$. Hence, $(k'_z)_{z\in Z}\in T(y)$. Our final task to obtain the required contradiction is to show that $\supp((k'_z))<\supp((k_z)).$ This is clear from \eqref{k'definition} however as we defined $k^- = \min_{z\in K_1}k_z$. We have obtained the required contradiction as $(k_z)$ was chosen to minimise $\supp((k_z))$ over all sequences in $T(y)$. Hence, $K$ must be dissociated.
\end{proof}
We continue with the proof of Proposition \ref{addispanineqprop}.
By Lemma \ref{dissK}, every $y\in \Sigma(Z)$ can be written as \begin{equation}
    y=\sum_{z\in K}k_z(y)z
    \label{yreprmini}
\end{equation} for some dissociated set $K=K(y)\subset Z$ and with $\sum_{z\in K}k_z(y)\leqslant a(y)\leqslant |Z|$. Let us write $X$ for the set of sequences $(n_z)_{z\in Z}\in \mathbf{N}^{Z}$ whose support is a dissociated subset of $Z$ and with $\sum_z n_z\leqslant |Z|$. We upper bound the size of $X$. Pick any sequence $(n_z)_{z\in Z}\in \mathbf{N}^{Z}$ in $X$ and let $N$ be its support. So $N$ is dissociated and as the largest dissociated subset of $Z$ has size $\dim(Z)$, we can fix a set $N'$ containing $N$ and of size exactly $\dim(Z)$. Then there are at most
\begin{align*}
    {|Z|\choose \dim(Z)}
\end{align*}
choices of $N'$ over all sequences in $X$. Given a set $N'$, it is a standard combinatorial fact that the number of sequences $(m_z)_{z\in N}\in \mathbf{N}^{N'}$ with $\sum_{z\in N'}m_z\leqslant |Z|$ is 
$${|Z|+|N'|\choose |N'|}={|Z|+\dim(Z)\choose \dim(Z)},$$
and clearly the sequence $(n_z)_{z\in Z}\in \mathbf{N}^{Z}$ is counted here as its support $N$ is contained in $N'$.
Hence, in total we get that \begin{equation}|X|\leqslant {|Z|\choose \dim(Z)}{|Z|+\dim(Z)\choose \dim(Z)}.
\label{Xupper}
\end{equation}
On the other hand, every $y\in \Sigma(Z)$ gives rise to the sequence $(k_z(y))_{z\in Z}$ as in \eqref{yreprmini} whose support is dissociated and with $\sum_z k_z(y)\leqslant |Z|$. Hence, $(k_z(y))_{z\in Z}$ in $X$. As $\sum_z k_z(y)z = y$ holds in $G$, no two distinct $y,y'\in \Sigma(Z)$ can give rise to the same sequence $(k_z(y))_{z\in Z}$ so that \begin{equation}|X|\geqslant \left|\Sigma(Z)\right|.
\label{Xlower}\end{equation}
Combining inequalities \eqref{Xupper} and \eqref{Xlower} yields the desired result \eqref{addispanineq}.
\end{proof}
To conclude this section, we simplify the bound obtained in Proposition \ref{addispanineqprop} to obtain Corollary \ref{addispanineqcor}.
\begin{proof}[Proof of Corollary \ref{addispanineqcor}]
    We use the following standard inequality for binomial coefficients with integers $0\leqslant r\leqslant n$:
     \begin{align*}
        {n+r\choose r}\leqslant \left(\frac{e(n+r)}{r}\right)^r,
    \end{align*}
    where $e$ is Euler's constant.
    Using inequality \eqref{addispanineq} and the inequality above with $n=|Z|$ and $r=\dim(Z)=d$ yields the desired bound
    \begin{align*}
    \left|\Sigma(Z)\right|&\leqslant {|Z|\choose d}{|Z|+d\choose d}\\
    &\leqslant {|Z|\choose d}{2|Z|\choose d}\\
    &\leqslant \left(\frac{e|Z|}{d}\right)^{d}\left(\frac{2e|Z|}{d}\right)^{d}\leqslant 2^{2d\log_2\left(\frac{4|Z|}{d}\right)}.
    \end{align*}
\end{proof}

\section{Balanced sets}
Let $A\subset G$ be a subset having no unique sum. To prove Theorem \ref{lowerboundthm}, we want to use Proposition \ref{addispanineqprop} with $Z$ = $A$ in order to deduce a lower bound on its size $|A|$. The first step in our proof of Theorem \ref{lowerboundthm} is therefore to show that the additive span $\Sigma(A)$ is large. It turns out that this step works under a weaker assumption than that $A$ has no unique sum, and this weaker property is all that is needed for this section.
\begin{definition}
\normalfont
    Let $B\subset G$ be a subset of an Abelian group $G$. We say that $B$ is \textit{balanced} if for every $b\in B$, there exist distinct $b_1,b_2\in B$ so that $2b=b_1+b_2$. In other words, $B$ is balanced if every element is the midpoint of a non-trivial 3-term arithmetic progression which is contained in $B$. If $B$ is balanced but does not contain two disjoint balanced subsets, then we say that $B$ is an \textit{irreducible balanced} set.
\end{definition}
Note that no finite subset of the integers is balanced, since the largest element is clearly not the midpoint of a non-trivial 3-term arithmetic progression contained in the set. Lemma \ref{rectificationp} therefore shows that if $B\subset \mathbf{Z}/p\mathbf{Z}$ is balanced, then $|B|\geqslant \log_2 p$ as $B$ cannot be rectifiable.
\begin{definition}
\normalfont
Let $G$ be a finite Abelian group. Then we define $b(G)$ to be the size of the smallest subset of $G$ which is balanced.
We also write $b(p)$ for $b\left(\mathbf{Z}/p\mathbf{Z}\right)$.
\end{definition}
Balanced sets have been studied in their own right in multiple papers, resulting in a very precise asymptotic for $b(p)$ which is correct up to lower order terms. In fact, $b(p)=(1+o(1))\log_2p$ where the lower bound follows from the rectification argument above and the upper bound comes from a construction of Nedev \cite{nedev}.
It is clear that if $A\subset G$ has no unique sum, then certainly the sum $a+a = 2a$ has a different representation as a sum of two elements in $A$ so $A$ is also a balanced set. Since balanced sets of size $(1+o(1))\log_2 p$ exist, the rectification bound is in a sense the only obstruction preventing a set from being balanced. For sets having no unique sum there are further obstructions, as the proof of Theorem \ref{generallowerboundthm} will show. 
\bigskip

From this section onward, all sets that we consider are proper sets (as opposed to multisets). Recall that for a proper set $S\subset G$, its additive span is simply the set of subset sums $$\Sigma (S)\vcentcolon = \left\{\sum_{s\in S'} s:S'\subseteq S\right\},$$
and we say that $S$ is an additive basis for $G$ if $\Sigma(S) = G$. The following proposition, whose proof we postpone to the end of this section, states that balanced sets have a translate with large additive span. 
\begin{proposition}
    Let $B\subset \mathbf{Z}/p\mathbf{Z}$ be a balanced set. Then there exists $g\in -B$ such that the translated set $B+g$ is an additive basis for $\mathbf{Z}/p\mathbf{Z}$.
\label{additivebasisp}
\end{proposition} 
As an aside, we note that this gives a new proof of the lower bound $b(p)>\log_2 p$ which, as we mentioned before, is best possible up to lower order terms.
\begin{cor}
    If $B\subset \mathbf{Z}/p\mathbf{Z}$ is balanced, then $|B|\geqslant \log_2 p +1$.
\label{balancedp}
\end{cor}
\begin{proof}
Let $B\subset \mathbf{Z}/p\mathbf{Z}$ be a balanced set, then there is some translate $B+g$ of $B$ so that $\Sigma(B+g) = \mathbf{Z}/p\mathbf{Z}$ and $g\in-B$ by Proposition \ref{additivebasisp}. As $0\in B+g$, we deduce that $\left|\Sigma(B+g)\right|\leqslant2^{|B|-1}$ and combining this with $\left|\Sigma(B+g)\right|\geqslant p$ yields the result.
\end{proof}

The situation in a general Abelian group $G$ is a bit more delicate. If $p$ is a prime dividing the order of a group $G$, then $G$ contains a cyclic group of order $p$ as a subgroup. Hence,
$$b(G)\leqslant \min_{\text{$p$ prime, }p| |G|} b(p).$$
However, one might hope that if $B\subset G$ is balanced and $B$ generates a large subgroup of $G$, then one can obtain an improved lower bound on $|B|$. This is not true in general as the property of being balanced is preserved under translation, so one could take a balanced subset of $G$ of size $b(p(G))$ and then take $B$ to be a translate of this set which generates a large subgroup of $G$. This is the reason for introducing the following definition. Here, for $C\subset G$ we use the standard notation $\langle C\rangle$ for the subgroup of $G$ generated by $C$.
\begin{definition}
\normalfont Let $C\subset G$, then we define $\text{minspan}(C)\vcentcolon= \min_{g\in G} |\langle C+g\rangle|$.
\end{definition}
Further, one can see that the union of any two balanced sets is balanced. This again could lead to small balanced sets in $G$ for which any translate generates a large subgroup of $G$. To avoid this, we work with irreducible balanced sets and in doing so, we obtain the following generalisation of Proposition \ref{additivebasisp}, and it is the only result from this section that we need for the proof of Theorem \ref{generallowerboundthm}.
\begin{proposition}
    Let $G$ be a finite Abelian group and let $B\subset G$ be an irreducible balanced set. Then there exists $g\in -B$ 
 such that $\Sigma(B+g)=\langle B+g\rangle$, i.e. the translated set $B+g$ is an additive basis for $\langle B+g\rangle$.
\label{additivebasis}
\end{proposition}
\begin{rmk*}
There do in fact exist non-irreducible balanced sets $B\subset G$ for which no translate $B+g$ is an additive basis for $\langle B+g\rangle$. As an example one can consider \begin{equation}
B=\mathbf{Z}/3\mathbf{Z}\times\{0,1\}\subset \mathbf{Z}/3\mathbf{Z}\times \mathbf{Z}/p\mathbf{Z}.
\label{balancedexample}
\end{equation} Then $B$ is balanced, but any translate contains an element of order $3p$, so $\langle B+g\rangle$ has size at least $3p$ and cannot have an additive basis of size $|B|=6$ for $p$ large. Note that in this example, $B$ is not irreducible as it is the disjoint union of two balanced sets of size $3$.
\end{rmk*}
We deduce the following corollary, giving an improved lower bound for sets with large $\text{\normalfont \text{minspan}}$.
\begin{cor}
    Let $G$ be an Abelian group. If $B\subset G$ is an irreducible balanced set, then $|B|\geqslant \log_2(\text{\normalfont \text{minspan}}(B))+1$.
\label{additivebasiscor}
\end{cor}
\begin{proof}
By Proposition \ref{additivebasis}, there is some translate $B+g$ of $B$ which contains $0$ and is an additive basis of $\langle B+g\rangle$. As $0\in B+g$, we deduce that $\left|\Sigma(B+g)\right|\leqslant 2^{|B+g|-1}=2^{|B|-1}$. As $B+g$ is an additive basis of $\langle B+g\rangle$, we also get that $\left|\Sigma(B+g)\right|\geqslant \left|\langle B+g\rangle\right|\geqslant \text{minspan}(B)$. Combining these two inequalities gives the result.
\end{proof}
If $B$ is a balanced set which is not irreducible, then this result breaks down. In fact, one cannot obtain any lower bound growing with $\text{\normalfont \text{minspan}}(B)$ without the assumption that $B$ is irreducible as the example defined in \eqref{balancedexample} shows. Hence, the following result is best possible for a general balanced set (up to lower order terms).
\begin{cor}
    If $B\subset G$ is balanced, then $|B|\geqslant \log_2 p(G) +1$.
\label{balancedG}
\end{cor}
\begin{proof}
Let $B\subset G$ be a balanced set, then it contains an irreducible balanced subset $B'\subset B$. Any balanced set clearly has at least two distinct elements, so any translate $B'+g$ contains a non-zero element of $G$. Hence, we see that $\langle B'+g\rangle$ is a non-zero subgroup of $G$ so has size at least $p(G)$ by Lagrange's theorem. So $\text{minspan}(B') \geqslant p(G)$ and using the result from Corollary \ref{additivebasiscor} gives the required lower bound $|B|\geqslant |B'|\geqslant \log_2p(G)+1$.
\end{proof}

\begin{rmk*}
The main purpose of this section is to prove the auxiliary result Proposition \ref{additivebasis} for the proof of Theorem \ref{generallowerboundthm}, but we stated some of its corollaries which are interesting in their own right as they yield a new strongest lower bound on the size of a balanced set in a general Abelian group:
$$|B|\geqslant \min\big(\log_2\text{\normalfont \text{minspan}}(B), \,2\log_2p(G)+1\big)+1.$$ This follows by applying Corollary \ref{additivebasiscor} if $B$ is irreducible, and applying Corollary \ref{balancedG} to the two disjoint balanced sets contained in $B$ that exist if $B$ is not irreducible.
\end{rmk*}
Let us now give the proof of Propositions \ref{additivebasisp} and \ref{additivebasis}. First we give the simple deduction of Proposition \ref{additivebasisp} from Proposition \ref{additivebasis}.
\begin{proof}[Proof of Proposition \ref{additivebasisp} assuming Proposition \ref{additivebasis}]
Let $p$ be prime and $B\subset \mathbf{Z}/p\mathbf{Z}$ be a balanced set. Note that $B$ contains an irreducible balanced subset $B'\subset B$. By Proposition \ref{additivebasis}, as $B'$ is an irreducible balanced set, there exists $g\in -B'\subset -B$ so that $B'+g$ is an additive basis for $\langle B'+g\rangle$. Any balanced set clearly has at least two distinct elements, so $B'+g$ contains a non-zero element of $\mathbf{Z}/p\mathbf{Z}$ so that $\langle B'+g\rangle = \mathbf{Z}/p\mathbf{Z}$. Hence,  $$\Sigma(B+g)\supset \Sigma(B'+g)\supset \mathbf{Z}/p\mathbf{Z}.$$
\end{proof}
\begin{proof}[Proof of Proposition \ref{additivebasis}.]
Let $G$ be a finite Abelian group and let $B\subset G$ be an irreducible balanced subset. We will show that there exists an element $g\in -B$ for which the translated set $B+g$ is an additive basis for the subgroup $\langle B+g\rangle\leqslant G$. We will pick such a $g\in -B$ later and then consider the translated set $B+g$. Clearly, $B+g$ is still an irreducible balanced set since this property is preserved under translation. Now let us pick any element $y\in \langle B+g\rangle$ and as the ambient group $G$ is finite, we can find non-negative integers $n_{b}(y)$ for $b\in B$ so that $y=\sum_{b\in B} n_{b}(y)(b+g)$. If it is the case that each such $n_{b}(y)$ is either $0$ or $1$, then we immediately deduce the desired conclusion that $y$ lies in the additive span $\Sigma(B+g)$ of $B+g$. If not, there is some $b\in B$ with $n_{b}(y)\geqslant 2$ and we will then use the relation $2b=b_1+b_2$ to decrease $n_{b}(y)$. By applying such `compressions' in a certain order, we obtain a way to write $y$ as a sum of the form $\sum_{b\in B} m_{b}(b+g)$ with $m_{b}\in\{0,1\}$ so that $y\in \Sigma(B+g)$ as desired. For this, we will use ´weight-compressions' which, for an expression $y=\sum_{b\in B} n_b (b+g)$ with non-negative integers $n_b$ that is not already maximally compressed, yield a new expression $y=\sum_b m_b (b+g)$ with larger weight. To define a weight function on the finite set $B$, it is convenient to use the language of graph theory.

\bigskip

Note that $B$ contains a balanced subset $B'$ which is minimal in the sense that no proper subset of $B'$ is balanced.\footnote{Recall that $B$ is irreducible, meaning that $B$ does not contain two disjoint balanced subsets. In general, an irreducible balanced set $B$ can still contain a proper subset $B'$ which is also balanced.} Let $H$ be a directed graph with vertex set $B$ and for each vertex $b\in B$ we have two outgoing edges $b\to b_1$ and $b\to b_2$ where $b_1,b_2\in B$ are distinct with $2b=b_1+b_2$. As $B$ is balanced, we can always find such $b_1,b_2$. If there is more than one choice of $b_1,b_2$ then we just pick one of them arbitrarily, except when $b\in B'$ in which case we always choose $b_1,b_2\in B'$. So every vertex in $H$ has outdegree exactly $2$ and every vertex in the induced subgraph $H[B']$ also has outdegree $2$. We need the following lemma.
\begin{lemma}
    Let $H$ be the directed graph defined above. Then for any vertex $g'\in B'$ and any vertex $b\in V(H)=B$, there is a directed path from $b$ to $g'$ in $H$.
\label{connectedlemma}
\end{lemma}
\begin{proof}
    Define for a vertex $h\in V(H)=B$ the set $R(h)$ to be the set of vertices $h_1$ in $H$ for which there exists a directed path (possibly consisting of a single vertex) from $h$ to $h_1$ in $H$. We show that $R(h)\subset B$ is itself a balanced set for all $h\in B$. Consider an element $x\in R(h)$ so there exists a directed path $P$ in $H$ going from $h$ to $x$. As $B$ is balanced, we can find distinct $x_1,x_2$ in $B$ with $2x=x_1+x_2$ and so that $x\to x_1$ and $x\to x_2$ are edges of $H$. But then concatenating $P$ with each of these edges gives directed paths from $h$ to $x_1$ and from $h$ to $x_2$. Hence, $x_1,x_2\in R(h)$ and we conclude that $R(h)$ is indeed a balanced set itself. This shows that for each $b'\in B'$, the set $R(b')\subset B'$  is a balanced subset of $B'$ so that as $B'$ was assumed to be a minimal balanced set, we get that $R(b')=B'$. Thus, we have shown that for any two vertices $b_1',b_2'\in B'$, there is a directed path from $b_1'$ to $b_2'$ in $H$. Finally, for any $b\in B$, the two sets $B'$ and $R(b)$ are balanced subsets of $B$. As $B$ is irreducible, this means that $B'$ and $R(b)$ intersect so there is a directed path from $b$ to a vertex in $B'$. We have shown that any two vertices in $B'$ are connected by a directed path so that $B'\subset R(b)$ as desired.
\end{proof}
We continue with the proof of Proposition \ref{additivebasis}. Pick any $g\in -B'$ and we will show that the translated set $B+g$ has the desired properties, meaning that $\Sigma(B+g)=\langle B+g\rangle$. Let $g'=-g$ so $g'\in B'$ and we are ready to define our weight function. For each vertex $b\in B$, let the number $s(b)$ denote the length of the shortest directed path from $b$ to $g'$ in $H$, so $s(g')=0$ for example. By Lemma \ref{connectedlemma}, $s(b)$ is finite for every $b$. We now define a weight function $w:B\to \mathbf{R}$ on $B$ by $w(b)\vcentcolon = 2^{-s(b)}$ for each $b\in B$. Let $y\in\langle B+g\rangle$ so we can write $y = \sum_{b\in B} n_{b}(y)(b+g)$ for some non-negative integers $n_{b}(y)$. Let $N_y = \sum_{b\in B} n_{b}(y)\in\mathbf{N}$ and consider the following set of $|B|$-tuples of non-negative integers:
$$S_{B,g}(y)\vcentcolon = \left\{(m_b)_{b\in B}\in \mathbf{N}^{B}: \sum_{b\in B} m_b(b+g) = y\text{ and } \sum_{b\in B} m_b=N_y\right\}.$$
For each $|B|$-tuple $(m_b)_{b\in B}$ in $S = S_{B,g}(y)$, we define its weight $$w\left((m_b)_{b\in B}\right)\vcentcolon= \sum_{b\in B} m_bw(b)\in [0,\infty).$$ Now $S$ contains the tuple $(n_{b}(y))_{b\in B}$ so it is certainly non-empty. Further, the number of $|B|-$tuples $(m_b)_{b\in B}\in \mathbf{N}^{B}$ with $\sum_b m_b= N_y$ is finite, so $S$ is a finite set. The idea is now to consider the tuple $(k_b)_{b\in B}$ in $S$ with maximal weight and we show that this forces each $k_b$ with $b\neq g'$ to be either $0$ or $1$. So let $(k_b)_{b\in B}$ be a tuple in $S$ with maximal weight and suppose for a contradiction that there is some $b\in B\setminus{\{g'\}}$ with $k_b\geqslant 2$. Then let $P = b,b_1,\dots,g'$ be a shortest path from $b$ to $g'$ in $H$, of length $s(b)\geqslant 1$. By definition of the edges in $H$ this means that there exists $b_2\in V(H)=B$ so that $2b=b_1+b_2$. Now let $(k'_c)_{c\in B}$ be a new tuple of non-negative integers defined by $k'_c = k_c$ for all $c\in B\setminus{\{b,b_1,b_2\}}$, $k'_b=k_b-2\geqslant 0$, and $k'_{b_i}=k_{b_i}+1$ for $i=1,2$. We show that $(k'_c)_{c\in B}\in S$. First note that $\sum_c k'_c = \sum_c k_c= N_y$. As $2b=b_1+b_2$, we also have that 
\begin{align*}
y &= y+(b_1+g)+(b_2+g)-2(b+g)\\
&=\Big(\sum_c k_c (c+g)\Big) +(b_1+g)+(b_2+g)-2(b+g)\\
&= \sum_c k'_c (c+g).
\end{align*}
So $(k'_c)\in S$, but we show that its weight is strictly larger than the weight of $(k_c)$ giving the required contradiction:
\begin{align*}
    w\left((k'_c)\right) &= \sum_c k'_c w(c) \\
    &= (k_b-2)w(b)+(k_{b_1}+1)w(b_1)+(k_{b_2}+1)w(b_2) +\sum_{c\in B\setminus{\{c,b_1,b_2\}}} k_cw(c) \\ 
    &= w\left((k_c)\right)-2w(b)+w(b_1)+w(b_2)\\
    &= w\left((k_c)\right)-2\cdot2^{-s(b)}+2^{-s(b_1)}+2^{-s(b_2)}\\
    &\geqslant w\left((k_c)\right)-2^{-s(b)+1}+2^{-s(b)+1}+2^{-s(b_2)}\\
    &>w\left((k_c)\right),
\end{align*}
where we used that $s(b_1)\leqslant s(b)-1$ as $b_1$ comes after $b$ in the shortest path $P$ from $b$ to $g'$, and that $w(b_2)=2^{-s(b_2)}>0$. 
\bigskip

We conclude that if $(k_b)_{b\in B}$ is an element of $S$ with largest weight, then $k_b \in \{0,1\}$ for all $b\in B\setminus{\{g'\}}$. Hence, the following equality shows that $y\in \Sigma(B+g)$ as desired:
\begin{align*}
    y &= \sum_{b \in B} k_b (b+g)\\
    &= \sum_{b\in B\setminus{\{g'\}}} k_b b + k_{g'}(g'+g)\\
    &= \sum_{b\in B\setminus{\{g'\}}} k_b (b+g) \in \Sigma(B+g),
\end{align*}
as we chose $g=-g'$. Since $y\in \langle B+g\rangle$ was arbitrary, we have shown that $\langle B+g\rangle=\Sigma(B+g)$.
\end{proof}
\begin{rmk*}
    Note that in the proof, the compressions can be used repeatedly on the original sum $y=\sum_b n_b(b+g)$ until we obtain a sum $y=\sum_bk_b(b+g)$ for which almost all the weight is placed at the vertex $g'\in V(H)=B$. Since this element contributes $k_{g'}(g'+g)= 0\in G$ to the sum, it does not matter that $k_{g'}$ is generally not in $\{0,1\}$. This shows the advantage of working with a translate $B+g$ instead of $B$, and in fact this is necessary as $B$ does not need to be an additive basis for $\langle B\rangle$.
\end{rmk*}
\begin{rmk*}
    Consider an arithmetic progression $Q=\{a,2a,\dots ka\}$of length $k$ in $G=\mathbf{Z}/p\mathbf{Z}$. Then every element of $Q$ except $a$ and $ka$ is the midpoint of a non-trivial 3-term arithmetic progression in $Q$. Taking $k=200$ for example, one can get that that $Q$ is `almost' balanced, in the sense that 99\% of the elements $b\in Q$ have that $2b=b_1+b_2$ for distinct $b_1,b_2\in Q$. However, $Q$ is very far from being balanced in that one needs to add at least $\log_2p - 200$ more elements to make it balanced. This observation shows that any proof of a logarithmic lower bound must have an algebraic flavour in the sense that one must crucially use that every single $b$ satisfies a balanced relation, as opposed to almost every $b$.
\end{rmk*}
We finish this section on balanced sets by using them to construct small sets in $\mathbf{Z}/p\mathbf{Z}$ having no unique sum, thus proving our new upper bound on $m(G)$ from Theorem \ref{upperboundthm}. In order to construct a set $A\subset G$ having no unique sum, it is natural to try using sets with a gridlike structure. Indeed, let $C,D$ be any subsets of the finite Abelian groups $G$ and $G'$ respectively. Then consider the Cartesian product $C\times D\subset G\times G'$ and let $(c_1,d_1)+(c_2,d_2)$ be any sum in its sumset $C\times D+C\times D$. Then we have that \begin{equation}
    (c_1,d_1)+(c_2,d_2) = (c_1,d_2)+(c_2,d_1).
    \label{gridsum}
\end{equation} This trivial observation shows that such a sum can only be unique if $c_1=c_2$ or if $d_1=d_2$. To fix the fact that sums as in \eqref{gridsum} where $c_1=c_2$ or $d_1=d_2$ can be unique in general, we consider the set $A=B\times B\subset (\mathbf{Z}/p\mathbf{Z})^2$ for a balanced set $B\subset \mathbf{Z}/p\mathbf{Z}$. It is easy to check that $A$ has no unique sum. Let $b,b',c,c'\in B$, so the sum $(b,b')+(c,c')$ is not unique if $b\neq c$ and $b'\neq c'$ by \eqref{gridsum}. On the other hand, if $b=c$ then we can write $b+c=2b=b_1+b_2$ for distinct $b_1,b_2\in B$ as $B$ is balanced. Then the equality $(b,b')+(c,c')= (b_1,b')+(b_2,c')$ shows that the sum is not unique. Finally, the same argument shows that $(b,b')+(c,c')$ is not unique if $c=c'$. Using this observation and \cite{nedev}, Theorem \ref{upperboundthm} easily follows.

\begin{proof}[Proof of Theorem \ref{upperboundthm}]
By Theorem 1 in \cite{nedev}, one can find a balanced set $B\subset\mathbf{Z}/p\mathbf{Z}$ of size $$|B|\leqslant (1+o(1))\log_2 p.$$ From the paragraph above, $A=B\times B\subset(\mathbf{Z}/p\mathbf{Z})^2$ is a set having no unique sum of size $|A|=|B|^2$. It is clear that if $T$ and $T'$ are Freiman-isomorphic sets, then $T$ has no unique sum if and only if $T'$ has no unique sum. Hence, all that remains is to find a Freiman isomorphism from $A\subset(\mathbf{Z}/p\mathbf{Z})^2$ into $\mathbf{Z}/p\mathbf{Z}$. Let $r\in \mathbf{Z}/p\mathbf{Z}$ and define $\phi_r: A\to \mathbf{Z}/p\mathbf{Z}: (b,b')\mapsto b+rb'$. Then $\phi_r$ is a Freiman homomorphism for all values of $r$, and can only fail to be a Freiman isomorphism if $r\in (2B-2B)/(2B-2B)$. As this set contains at most $|B|^8\leqslant(1+o(1))(\log_2p)^8<p$ elements for $p$ large, the map $\phi_r$ is a Freiman isomorphism for some $r$, thus giving a subset of $\mathbf{Z}/p\mathbf{Z}$ of size $|B|^2=O((\log p)^2)$ with no unique sum. 
\end{proof}
\begin{rmk*}
    One can improve the implied constant by a factor of $2$ by noting that if $B$ is a balanced set in $\mathbf{Z}/p\mathbf{Z}$, then the set $A=B+B\subset\mathbf{Z}/p\mathbf{Z}$ has no unique sum and size at most ${|B|+1 \choose 2}$. We opted to give the proof based on a Freiman isomorphism from $B\times B$ into $\mathbf{Z}/p\mathbf{Z}$ as it makes clear that we are using a certain two-dimensional structure in order to get non-unique sums. In correspondence with Kopparty, the author found out that essentially the same construction for this upper bound also appears in the thesis of Scheinerman \cite{danny}.
\end{rmk*}

\section{Sets with no unique sum have small dimension}
In this final section, we combine all our results to prove a lower bound on the size of a set $A\subset G$ having no unique sum. Our argument begins by applying the inequality in Proposition \ref{addispanineqprop} to $A+g$ and plugging in the lower bound on $\left|\Sigma(A+g)\right|$ from Proposition \ref{additivebasis}. We introduce the following convenient notation.
\begin{definition}
    Let $Z\subset G$, then define the number
    \begin{equation}
        K(Z)\vcentcolon=  \frac{|Z|}{\dim(Z)}.
    \end{equation}
\end{definition}
Then simply rewriting the inequality in Corollary \ref{addispanineqcor} using that $\dim(Z) = \frac{|Z|}{K(Z)}$ gives the following.
\begin{proposition}
     Let $G$ be an Abelian group and let $Z$ be a finite multiset consisting of elements of $G$. Then
    \begin{equation}
    |Z|\geqslant \frac{K(Z)}{2\left(2+\log_2 K(Z)\right)}\cdot \log_2 \left|\Sigma(Z)\right|.
    \label{lowerboundaddispan}
    \end{equation}
\label{lowerboundaddispancor}
    
\end{proposition}
This is the form of the inequality that will be useful for our purpose. In the previous section, we have proven Proposition \ref{additivebasis} which showed that if $B\subset G$ is a balanced set, then it contains an irreducible balanced subset $B'$ which has a translate $B'+g$ that forms an additive basis for $\langle B'+g\rangle$. Since $B'+g$ contains a non-zero element in $G$, $\langle B'+g\rangle$ is a non-trivial subgroup of $G$ so that by Lagrange's Theorem we get \begin{equation}
\left|\Sigma(B+g)\right|\geqslant \left|\Sigma(B'+g)\right| \geqslant\left|\langle B'+g\rangle \right|\geqslant p(G),
\end{equation}
where $p(G)$ denotes the smallest prime divisor of $G$. Using this in inequality \eqref{lowerboundaddispan} shows that for any balanced set $B$ in an Abelian group $G$, we have the lower bound
\begin{equation}
    |B|=|B+g|\geqslant \frac{K(B+g)}{2\left(2+\log_2 K(B+g)\right)}\cdot \log_2 p(G).
\label{balanceddissineq}
\end{equation}
At this point we make an interesting observation. As we noted in the previous section, Nedev \cite{nedev} showed that $G$ contains a balanced subset of size at most $(1+o(1))\log_2p(G)$. Now looking at \eqref{balanceddissineq}, if $B$ is a balanced set  of size $C\log_2 p(G)$ then we must have that $K(B+g)=O_C(1)$ is bounded by a constant depending only on $C$. This means precisely that any such balanced set $B$ in $G$ has a translate with large additive dimension $\dim(B+g)=\frac{|B|}{K(B+g)}\gg_C |B|$, i.e. it contains a dense dissociated subset. The goal in this section is to show that the situation is different for sets having no unique sum. We show that if $A\subset G$ has no unique sum, then $A$ does not contain a dense dissociated subset and in fact we have $\dim(A) =o_{p(G)}(1)\cdot|A|$. Equivalently, we prove that $K(A)\to \infty$ as $p(G)\to \infty$, and plugging this into \eqref{balanceddissineq} will then give the desired lower bound
$$|A|\geqslant \omega(p(G))\log_2p(G).$$
\begin{rmk*}
Note that it is not true that $\dim(A)=o(|A|)$ as $|A|\to\infty$ for sets $A$ containing no unique sum. What we do show is that $\dim(A)=o_{p(G)}(1)\cdot |A|$. For an example, we can consider $A_1=B\times \mathbf{Z}/3\mathbf{Z}\subset \mathbf{Z}/p\mathbf{Z}\times \mathbf{Z}/3\mathbf{Z}$ where $B\subset \mathbf{Z}/p\mathbf{Z}$ is a balanced set of size $C\log_2p$ with $\dim(B)\geqslant\frac{|B|}{C'}$ (such $B$ exist by the discussion above). Then $A_1$ is a product of two balanced sets and hence has no unique sum. However, $A_1$ contains $B\times \{0\}$ so $\dim(A_1)\geqslant\dim(B)\geqslant\frac{|B|}{C'}\geqslant \frac{|A_1|}{3C'}$. By taking $p$ large, it follows that there are arbitrarily large sets having no unique sum but which do contain a dense dissociated subset. The issue in this example, of course, is that the smallest prime factor of $\left|\mathbf{Z}/p\mathbf{Z}\times \mathbf{Z}/3\mathbf{Z}\right|$ is bounded.
\end{rmk*}
The following proposition gives a precise statement of the fact that sets with no unique sum have small additive dimension.
\begin{proposition}
    Let $G$ be a finite Abelian group with $p(G)$ being the smallest prime dividing $|G|$. Let $A\subset G$ have no unique sum. Then
    \begin{equation}
        K(A)\geqslant \omega_1(p(G)),
    \label{Kinfty}
    \end{equation}
    for some function $\omega_1:\mathbf{N}\to (0,\infty)$ with $\omega_1(n)$ tending to infinity as $n\to \infty$. Moreover, one can take
    \begin{equation}
        \omega_1(n) \gg \sqrt{\log\log\log n},
    \end{equation}
    where the implied constant is absolute.
\label{disssmallprop}
\end{proposition}
Assuming this proposition for the moment, we can put everything together and complete the proof of Theorem \ref{generallowerboundthm}.
\begin{proof}[Proof of Theorem \ref{generallowerboundthm} assuming Proposition \ref{disssmallprop}]
Let $A\subset G$ have no unique sum. In particular, this means that $A$ is balanced so \eqref{balanceddissineq} gives that
\begin{equation}
    |A|\geqslant \frac{K(A+g)}{2\left(2+\log_2 K(A+g)\right)}\cdot \log_2 p(G)
\label{conclusionineq}
\end{equation}
for some element $g\in G$.
As $A\subset G$ does not have a unique sum, neither does the translated set $A+g$. By Proposition \ref{disssmallprop}, we then get that $$K(A+g)\geqslant \omega_1(p(G)).$$
Plugging this in \eqref{conclusionineq} gives
\begin{align*}
    |A|\geqslant \omega(p(G))\log_2 p(G)
\end{align*}
if we define
$$\omega(n) \vcentcolon= \frac{\omega_1(n)}{2(2+\log_2\omega_1(n))}.$$
Note that by Proposition \ref{disssmallprop}, we have that $\omega_1(n)\gg \sqrt{\log\log\log n}$ so
$$\omega(n) \gg \frac{\sqrt{\log\log \log n}}{\log\log\log \log n}.$$ This concludes the proof of Theorem \ref{generallowerboundthm}.
\end{proof}
\begin{rmk*}
Theorem \ref{generallowerboundthm} is rather delicate in the sense that the result is false if one only assumes that $A$ is balanced and that most sums in $A$ are not unique. To see this, consider $A_1=B\times Q\subset \left(\mathbf{Z}/p\mathbf{Z}\right)^2=G$ with $B$ a minimal balanced set and $Q$ an arithmetic progression of size $200$. Then $A_1$ is balanced as $B$ is, and $A_1$ has that $99\%$ of its sums are non-unique. However, $|A_1|\leqslant 200(1+o(1))\log_2p(G)$.
\end{rmk*}
Our final task, then, is to prove Proposition \ref{disssmallprop}. 
The main tool in the proof is the following proposition.
\begin{proposition}
    Let $G$ be a finite Abelian group and let $A\subset G$ have no unique sum. There exists an absolute constant $C$ such that the following holds. Suppose that $D$ is a dissociated subset of $A$ with $|D|\geqslant 10$ (say) and that $S\subset G$ contains $0$. If
    \begin{equation}
        |S|\leqslant \min\left(\log_2p(G),\left(\frac{|D|^6}{C|A|^5}\right)^{\frac{1}{4}}\right),
    \label{Sassumptions}
    \end{equation}
    then there exists a set $S'\subset G$ containing $0$ and of size
    \begin{equation}
        |S'|\leqslant \max(2|S|,|S|^3)
    \end{equation}
    so that
    \begin{equation}
    \left|(D+S')\cap A\right|\geqslant \left|(D+S)\cap A\right|+\frac{|D|^2}{36|A|}.
    \end{equation}
\label{iteratetran}
\end{proposition}
Roughly speaking, this proposition states that if a set $A$ having no unique sum contains a dissociated subset $D$ so that few translates of $D$ (namely the set $D+S$) contains a certain fraction of all elements of $A$, then there exists a slightly larger set of translates $S'$ so that $D+S'$ contains a significantly bigger fraction of $A$.
Assuming this proposition, we show how to deduce Proposition \ref{disssmallprop} using a density increment argument.
\begin{proof}[Proof of Proposition \ref{disssmallprop} assuming Proposition \ref{iteratetran}]
Let $A\subset G$ have no unique sum, and let $D\subset A$ be a dissociated subset of $A$ of largest possible size. So $|D|=\dim(A)$ and $|D|=\frac{|A|}{K(A)}$.  Let $p=p(G)$. Our goal is to prove that $K(A)\gg \sqrt{\log\log \log p}$. If $|D|<10$, then $K(A)\geqslant \frac{|A|}{10}\gg \log p$ by Corollary \ref{balancedG} as $A$ is balanced, so we are done. Hence, we may assume that $|D|\geqslant 10$ so that $D$ satisfies the assumption of Proposition \ref{iteratetran}. We apply Proposition \ref{iteratetran} with $D\subset A$ and $S=S_0\vcentcolon=\{0\}$. Then either $|S_0|> \min\left(\log_2p,\left(\frac{|D|^6}{C|A|^5}\right)^{\frac{1}{4}}\right)$ or else the assumption \eqref{Sassumptions} is satisfied and we deduce that there exists a set $S_1\subset G$ containing $0$ of size at most $\max(2|S_0|,|S_0|^3)=2$ so that $\left|(D+S_1)\cap A\right|\geqslant \left|(D+S_0)\cap A\right|+\frac{|A|}{36K(A)^2}$. Suppose that after $i$ steps we have a set $S_i\subset G$ containing $0$ and of size at most
$$|S_i|\leqslant 2^{3^i}$$
so that
$D+S_i$ contains at least $\frac{i|A|}{36K(A)^2}$ of the elements in $A$. Then either we have that $|S_i|>\min\left(\log_2p,\left(\frac{|D|^6}{C|A|^5}\right)^{\frac{1}{4}}\right)$ or else \eqref{Sassumptions} is satisfied so we can again apply Proposition \ref{iteratetran} to find a set $S_{i+1}\subset G$ also containing $0$ and of size
$$|S_{i+1}|\leqslant \max(2|S_i|,|S_i|^3)\leqslant 2^{3^{i+1}}$$
so that $D+S_{i+1}$ contains a fraction of at least $\frac{i+1}{36K(A)^2}$ of the elements of $A$. Now it is clear that $\left|(D+S_i)\cap A\right|\leqslant |A|$ for all $i$, and hence the iterative procedure described above must fail for some $j\leqslant 36K(A)^2$. By Proposition \ref{iteratetran}, the only way that this can happen is if $|S_j|>\min\left(\log_2p,\left(\frac{|D|^6}{D|A|^5}\right)^{\frac{1}{4}}\right)$. First, if $|S_j|>\log_2p$, then as $|S_j|\leqslant 2^{3^j}$ we deduce that $36K(A)^2\geqslant j\gg \log\log |S_j|\gg \log\log\log p$ and we have proven \eqref{Kinfty}.
\smallskip

Finally, if $|S_j|>\left(\frac{|D|^6}{C|A|^5}\right)^{\frac{1}{4}}$, then recalling that $|S_j|\leqslant 2^{3^j}$ and that we defined $|D|=\frac{|A|}{K(A)}$ gives
 \begin{align*}
     36K(A)^2&\geqslant j \gg \log\log |S_j|\geqslant \log\Bigg(\frac{1}{4}\log\left(\frac{|A|}{CK(A)^6}\right)\Bigg)\\
     &= \log\Big(\frac{1}{4}\log_2(|A|)-\frac{1}{4}\log_2(CK(A)^6)\Big)\\
     &\geqslant \log\Big(\frac{1}{4}\log_2\log_2p-\frac{1}{4}\log_2(CK(A)^6)\Big),
 \end{align*}
where here we used the weak bound $|A|\geqslant \log_2p(G)$ which by Corollary \ref{balancedG} holds even for balanced sets.
We deduce that $K(A)\gg \sqrt{\log\log\log p}$ as desired.
\end{proof}
To complete the proof, it now only remains to prove Proposition \ref{iteratetran}. So far, we have not yet used that $A$ has no unique sum but only the much weaker assumption that $A$ is balanced. As we saw, the conclusion of Proposition \ref{disssmallprop} fails completely for a general balanced set, even if most sums are not unique. So our proof of Proposition \ref{iteratetran} here must inevitably use that $A$ has no unique sum. This will make the proof somewhat technical, but this seems unavoidable at this stage. We begin with a useful lemma.
\begin{lemma}
    Let $G$ be a finite Abelian group and let $S\subset G$ have size at most $\log_2p(G)$. Then  we can assign an element $s_X\in X$ to each non-empty subset $X\subset S$ such that the following holds. Let $X,Y\subset S$ be non-empty subsets, then the only solution to $x+y=s_X+s_Y$ with $x\in X$ and $y\in Y$ is the trivial solution $(x,y)=(s_X,s_Y)$.
\label{simulrect}
\end{lemma}
\begin{proof}
    $S$ is rectifiable by Lemma \ref{rectificationG} and let $\phi:S\to S'\subset \mathbf{N}$ be a Freiman isomorphism to a subset of $\mathbf{N}$. Let $X\subset S$ be non-empty. We claim that the choice $s_X=\phi^{-1}(\max \phi(X))$ works. Let $X,Y\subset S$ be non-empty. As $\phi$ is a Freiman isomorphism, it is enough to show that the only solution $(x,y)\in \phi(X)\times \phi(Y)$ to $x+y=\max\phi(X)+\max \phi(Y)$ is the trivial one. This is clear since $x+y\leqslant \max\phi(X)+\max \phi(Y)$ is an inequality in the integers, where equality only holds if $(x,y)=(\max\phi(X),\max \phi(Y))$.
\end{proof}
We are now ready to begin the proof of Proposition \ref{iteratetran}.
\begin{proof}[Proof of Proposition \ref{iteratetran}]
Let $G$ be a finite Abelian group and let $A\subset G$ have no unique sum. Suppose that $D$ is a dissociated subset of $A$. Let $S\subset G$ contain $0$ and have size $|S|\leqslant\min\left( \log_2 p(G),\left(\frac{|D|^6}{C|A|^5}\right)^{\frac{1}{4}}\right)$. Finally, suppose that the set $D+S$ contains a fraction $\alpha$ of all elements of $A$, meaning that
\begin{equation}
    \left|(D+S)\cap A\right| = \alpha|A|.
\end{equation}
The idea is that since $D$ is dissociated and the set of shifts $S$ is small, the set $(D+S)\cap A$ still has many unique sums. Since $A$ has no unique sum, we will show that this forces $A$ to contain a large part of a larger set of translates $D+S'$.
\bigskip

We introduce some notation. For each $d\in D$, we define 
\begin{equation}
    S_d=\{s\in S:d+s\in A\}
    \label{S_ddefin}
\end{equation} and note that $0\in S_d$ as $D$ is a subset of $A$. Hence, each $S_d$ is a non-empty subset of $S$ and by applying Lemma \ref{simulrect}, we can find an element $s_d\in S_d$ for each $d\in D$ so that whenever $d,d'\in D$, then the only solution $(x,y)\in S_d\times S_{d'}$ to
\begin{equation}
    x+y=s_d+s_{d'}
    \label{s_duniquesum}
\end{equation}
is the trivial solution $(x,y)=(s_d,s_{d'})$. Also, let us define the following set of elements of $D$:
\begin{equation}
    B^{(1)}(D)\vcentcolon=\left\{d\in D: \text{there is a $v\in (2S-2S)\setminus{\{0\}}$ so that $d+v\in D$}\right\}.
\label{B^1defin}
\end{equation}
We think of $B^{(1)}(D)$ as the set of `bad' elements in $D$ as they will not be useful for our later argument. We shall prove later that this set $B^{(1)}(D)$ of bad elements of $D$ is rather small and hence we can simply remove it from $D$, so we will work with the set $G^{(1)}(D)\vcentcolon= D\setminus{B^{(1)}(D)}$ of `good' numbers. It turns out that $G^{(1)}(D)$ can still contain some bad pairs which leads to the following definition. Let $B^{(2)}(D)$ be the set of unordered pairs $\{d,d'\}\in {G^{(1)}(D)\choose 2}$ for which there exist $e,e'\in D$ and $s,s'\in S$ so that 
\begin{equation}
    d+s_d+d'+s_{d'}=e+s+e'+s'
\label{s_dequal}
\end{equation} 
and $\{e,e'\}\neq \{d,d'\}$.\footnote{ Here, we use the notation ${E \choose 2}$ to denote $\{\{x,y\}:x,y\in E, x\neq y\}$, the set of unordered pairs of  elements of $E$.}
We think of $B^{(2)}(D)$ as the set of `bad' pairs in ${G^{(1)}(D)\choose 2}$ because one can see how if \eqref{s_dequal} holds, then the sum $(d+s_d)+(d'+s_{d'})\in (D+S)+(D+S)$ does not yield unique sum in $D+S$. We will see later that $B^{(2)}(D)$ is also not too large so we can also remove such pairs. Hence, we define the set of `good' pairs in ${D\choose 2}$ as follows:
\begin{equation}
    G^{(2)}=G^{(2)}(D)\vcentcolon={G^{(1)}(D)\choose 2}\setminus B^{(2)}(D).
\end{equation}
The following lemma shows what we mean by a pair $\{d,d'\}\in G^{(2)}$ being `good', namely that $(d+s_d)+(d'+s_{d'})$ is a unique sum in $(D+S)\cap A$.
\begin{lemma}
Let $\{d,d'\}\in G^{(2)}(D)$. Then the only solutions $(x,y)\in \left((D+S)\cap A\right)^2$ to
\begin{equation}
x+y=(d+s_d)+(d'+s_{d'})
\label{Gunique}
\end{equation}
are the trivial ones $(x,y) = (d+s_d,d'+s_{d'}),(d'+s_{d'},d+s_d)$.
\label{goodlemma}
\end{lemma}
\begin{proof}
Let $\{d,d'\}\in G^{(2)}(D)$ and suppose that $(x,y)\in \left((D+S)\cap A\right)^2$ is a solution to
$x+y=(d+s_d)+(d'+s_{d'})$. We show that $(x,y)$ is a trivial solution. As $x,y\in(D+S)\cap A$, there exist $e,e'\in D$ and $s,s'\in S$ so that $x=e+s$ and $y=e'+s'$. Further note that as $x,y\in A$, this means that $s\in S_e$ and $s'\in S_{e'}$ by the definition \eqref{S_ddefin} of the sets $S_e,S_{e'}$. We have that \begin{equation}
    d+s_d+d'+s_{d'}=x+y=e+s+e'+s'
    \label{uniquesumsD+S}
\end{equation} so we have an equation of the form \eqref{s_dequal} and recalling that $\{d,d'\}\in  G^{(2)}(D)$ is not in $B^{(2)}(D)$, we must then have that $\{d,d'\}=\{e,e'\}$. By reordering $x$ and $y$ (which does not affect whether or not $(x,y)$ is a trivial solution) we may therefore assume that $d=e$ and that $d'=e'$. Plugging this into \eqref{uniquesumsD+S} gives that $s_d+s_{d'}=s+s'$. But as $s\in S_d$ and $s'\in S_{d'}$ and as $s_d,s_{d'}$ were chosen to satisfy the conclusion of Lemma \ref{simulrect}, this is only possible if $s=s_d$ and $s'=s_{d'}$. Hence, $x=e+s=d+s_d$ and $y=e'+s'=d'+s_{d'}$ is a trivial solution to \eqref{Gunique}, as desired.
\end{proof}
We have shown the desirable property that pairs in $G^{(2)}(D)$ yield unique sums, and we now show that we have not removed too many elements in constructing $G^{(2)}(D)$ from ${D \choose 2}$, i.e. there are few `bad' pairs. Here, and several more times in the proof it will be convenient to have the following lemma at our disposal. It is an easy consequence of the skew version of Bollob\'as's Two Families Theorem (see for example \cite{frankl}), although as all sets involved have size at most 2 one could also give an elementary direct proof.
\begin{lemma}[Two Families Theorem]
There exists an absolute constant $C_1$ such that the following holds. Let $P_1,P_2,\dots,P_k$ and $Q_1,Q_2,\dots,Q_k$ be sets of size $n\leqslant 2$ so that 
\begin{itemize}
    \item $P_i\cap Q_i=\emptyset$ for all $1\leqslant i\leqslant k$.
    \item For every $i,j\in \{1,\dots,k\}$ with $i\neq j$ we have that $P_i\cap Q_j\neq \emptyset$ or $P_j\cap Q_i\neq\emptyset$.
\end{itemize}
Then $k\leqslant C_1$.
\label{bollobas}
\end{lemma}
The first important point is that $B^{(1)}(D)$ is rather small since $D$ is dissociated. 
\begin{lemma}
    We have that
    \begin{equation}
        \left|B^{(1)}(D)\right|\leqslant C_1|S|^4.
        \label{B1small}
    \end{equation}
\end{lemma}
\begin{proof}
Let us prove this claim by assuming for a contradiction that $\left|B^{(1)}(D)\right|> C_1|S|^4>C_1\left|(2S-2S)\setminus{\{0\}}\right|$. From \eqref{B^1defin}, we deduce that there must be some non-zero $v\in (2S-2S)$ so that for at least $C_1+1$ distinct elements $d_i\in D$, we have that $d_i+v\in D$ for $1\leqslant i\leqslant C_1+1$. Hence we can find $e_i\in D$ so that $d_i+v= e_i$ and note that \begin{equation}
    d_i\neq e_i
    \label{dnote1}
\end{equation} as $v\neq 0$.
Hence, we can apply Lemma \ref{bollobas} with $P_i=\{d_i\}$ and $Q_i=\{e_i\}$ to find distinct $i,j$ with $d_i\neq e_j$ and $d_j\neq e_i$. Without loss of generality, let $i=1,j=2$ so we deduce that
$$d_1+e_2=d_1+v+d_2=e_1+d_2$$ and we have shown that $\{d_1,e_2\},\{e_1,d_2\}$ are two subsets of $D$ with equal sum.\footnote{To show that $\{d_1,e_2\},\{d_2,e_1\}$ are subsets of $D$ as opposed to a multisubsets, we used the skew Two Families Theorem. For this application, a very simple direct argument would have sufficed, but we will make similar applications of the Two Families Theorem later and hence try to do this in a consistent manner.} As $D$ is dissociated, these sets are equal so by \eqref{dnote1} we must have that $d_1=d_2$. This is the required contradiction as $d_1,d_2$ were distinct. Hence, $\left|B^{(1)}(D)\right|\leqslant C_1|S|^4$ as desired.
\end{proof}
We now show that $B^{(2)}(D)$ is also not too large.
\begin{lemma}
    We have that
    \begin{equation}
    \left|B^{(2)}(D)\right|\leqslant C_1|S|^4+|D||S|^4
\label{smallbaad}
\end{equation}
\end{lemma}
\begin{proof}
Assume for a contradiction that $\left|B^{(2)}(D)\right|> C_1|S|^4+|D||S|^4$. For every pair $\{d,d'\}\in B^{(2)}(D)$ we can find $e,e'\in D$ with $\{e,e'\}\neq\{d,d'\}$ and $s,s'\in S$ such that \eqref{s_dequal} holds. So to each pair $\{d,d'\}\in B^{(2)}(D)$, we can associate a $4-$tuple $(s_d,s_{d'},s,s')\in S^4$ (there may be more than one choice of such a tuple, but in this case we pick one arbitrarily). As we are assuming for a contradiction that $\left|B^{(2)}(D)\right|> C_1|S|^4+|D||S|^4$, there must be some such $4-$tuple in $S^4$ that is associated to $C_1+|D|+1$ distinct pairs $\{d_i,d_i'\}\in B^{(2)}(D)$. Let $e_i,e_i'\in D$ and $s_i,s_i'\in S$ be so that
\begin{equation}
d_i+s_{d_i}+d_i'+s_{d_i'}=e_i+s_i+e_i'+s_i'
\label{s_dequali}
\end{equation}
with 
\begin{equation}
    \{e_i,e_i'\}\neq\{d_i,d_i'\} 
    \label{dnote}
\end{equation} for $i=1,\dots,C_1+|D|+1$, where $e_i,e'_i,s_i,s'_i$ exist by definition of $B^{(2)}(D)$. Then the assumption that all $\{d_i,d_i'\}$ are associated to the same $4-$tuple in $S^4$ means precisely that there exist fixed $s_d,s_{d'},s,s'\in S$ so that   \begin{equation}
s_{d_i}=s_{d},\,s_{d_i'}=s_{d'},\,s_i=s,\, s_i'=s'
\label{sametuple}
\end{equation} for all $i$. Clearly there are at most $|D|$ of these indices $i$ for which $e_i=e'_i$, so suppose without loss of generality that $e_i\neq e'_i$ for $i=1,\dots,C_1+1$. We will apply Lemma \ref{bollobas} with $P_i=\{d_i,d_i'\}$ and $Q_i=\{e_i,e_i'\}$ for $i=1,\dots,C_1+1$ and we first show that the condition $P_i\cap Q_i=\emptyset$ is satisfied. Indeed, if $d_i=e_i$ for a contradiction (the other cases when $P_i\cap Q_i\neq\emptyset$ can be handled similarly), then \eqref{s_dequali} would imply that $d_i'+(s_{d_1}+s_{d_1'}-s_1-s_1')=e_i'$.  In other words, $d_1'+v=e_1'$ for some $v\in 2S-2S$ and note that $v\neq 0$ as else $d_1'=e_1'$ but, as $d_1=e_1$, this would contradict \eqref{dnote}. However, the existence of a non-zero $v\in 2S-2S$ so that $d_1'+v\in D$ means precisely that $d_1'\in B^{(1)}(D)$ and this is impossible because we removed $B^{(1)}(D)$ from $D$ to obtain $G^{(1)}(D)$. Lemma \ref{bollobas} therefore gives distinct $i,j$ such that $P_i\cap Q_j=\emptyset=P_j\cap Q_i$ and without loss of generality, let $i=1,j=2$. Plugging \eqref{sametuple} in \eqref{s_dequal} then gives $$d_1+d_1'-e_1-e_1' = s_1+s_1'-s_{d_1}-s_{d_1'}=s_2+s_2'-s_{d_2}-s_{d_2'}= d_2+d_2'-e_2-e_2'.$$ So we obtain two subsets $Q=\{d_1,d_1',e_2,e_2'\}$ and $R=\{e_1,e_1',d_2,d_2'\}$ of the dissociated set $D$ having equal sum, where we noted that these are not multisets as $P_1\cap Q_2=\emptyset=P_2\cap Q_1$ by our application of Lemma \ref{bollobas} and that $e_i\neq e'_i$ for $i=1,\dots, C_1+1$. We conclude that $Q=R$. As  $\{d_1,d_1'\},\{d_2,d_2'\}$ were distinct pairs, we may (after potentially relabeling these elements) assume that $d_1\notin \{d_2,d_2'\}$. Then $d_1\in Q\setminus\{d_2,d_2'\} =R\setminus\{d_2,d_2'\}=\{e_1,e_1'\}$ which gives the desired contradiction as we showed above that $P_i\cap Q_i=\emptyset$ for all $i$.
\end{proof}

From \eqref{B1small} and \eqref{smallbaad} we deduce that the set $G^{(2)}$ of good pairs is still large:
\begin{align}
    \left|G^{(2)}\right|&\geqslant {\left|G^{(1)}(D)\right|\choose 2}-\left|B^{(2)}(D)\right|\nonumber\\
    &= {\left|D\right|-\left|B^{(1)}(D)\right|\choose 2}-\left|B^{(2)}(D)\right|\nonumber\\
    &\geqslant {|D|-C_1|S|^4\choose 2}-C_1\left|S\right|^4-|D||S|^4\nonumber\\
    &\geqslant \frac{|D|^2}{3},
\label{manygood}
\end{align}
where in the final line we used that $|D|\geqslant 10$ and the assumption \eqref{Sassumptions} so that $|S|^4\leqslant\frac{|D|^6}{C|A|^5}\leqslant\frac{|D|}{C}$ as $D\subset A$ so $|D|\leqslant |A|$, and we can take $C$ to be a sufficiently large constant. This result that there are many good pairs in $(D+S)\cap A$, i.e. many pairs giving a unique sum in $(D+S)\cap A$, is the only result out of all the work we did in this proof so far that will be needed for the rest of the argument.
\bigskip

For every pair $\{d,d'\}\in G^{(2)}$, we have that $(d+s_d)+(d'+s_{d'})$ is a sum in $A+A$ and therefore it must allow for a non-trivially different representation as a sum of two elements $x,y\in A$. By Lemma \ref{goodlemma}, it cannot be the case that both of $x,y$ lie in $(D+S)\cap A$, so we can define $x(d,d'),y(d,d')\in A$ so that $(d+s_d)+(d'+s_{d'})=x(d,d')+y(d,d')$ and $x(d,d')\in A\setminus{(D+S)}$.\footnote{Technically, one would have to write something like $x(\{d,d'\})$ as $\{d,d'\}$ is an unordered pair, but this is too cumbersome.} We introduce some further notation. Define for each $a\in A$ the set
\begin{equation}
    N(a)\vcentcolon=\left\{\{d,d'\}\in G^{(2)}: x(d,d')=a\right\}.
\label{N(a)definition}
\end{equation}
 so from \eqref{manygood} we obtain the inequality
\begin{equation}
    \sum_{a\in A\setminus{(D+S)}}\left|N(a)\right|=\left|G^{(2)}\right|\geqslant \frac{|D|^2}{3}
\label{largeN(a)sum}
\end{equation}
as each pair $\{d,d'\}\in G^{(2)}$ appears in exactly one $N(a)$ with $a\notin D+S$. We now pick out those $a\in A\setminus{(D+S)}$ for which $N(a)$ is large. We define
\begin{equation}
    \mathcal{N}\vcentcolon=\left\{a\in A\setminus{(D+S)}: |N(a)|\geqslant \frac{|D|^2}{6|A|}\right\}.
\label{mathcalNdefinition}
\end{equation}
We show that by a simple averaging argument, $\mathcal{N}$ is fairly large.
\begin{lemma}
    We have that
    $$\left|\mathcal{N}\right|\geqslant \frac{|D|^2}{6|A|}.$$
    \label{Nlarge}
\end{lemma}
\begin{proof}
First we prove that for every $a\in A$, we have the upper bound $\left|N(a)\right|\leqslant |A|$. In fact, we show that if $\{d_1,d_1'\},\{d_2,d_2'\}\in N(a)$ are distinct, then $y(d_1,d_1')\neq y(d_2,d_2')$. As $y(d,d')\in A$ always holds, there can then be at most $|A|$ pairs in $N(a)$. Now let $\{d_1,d_1'\},\{d_2,d_2'\}\in N(a)$ with $y(d_1,d_1')= y(d_2,d_2')$, then as $x(d_1,d_1')=x(d_2,d_2')=a$ we get that $$(d_1+s_{d_1})+(d_1'+s_{d_1'})=a+y(d_1,d_1')=a+y(d_2,d_2')=(d_2+s_{d_2})+(d_2'+s_{d_2'})$$ so that $\{d_1,d_1'\}=\{d_2,d_2'\}$ are not distinct by Lemma \ref{goodlemma}.
\smallskip

Using that $\left|N(a)\right|\leqslant |A|$ for $a\in \mathcal{N}$ and that $\left|N(a)\right|\leqslant \frac{|D|^2}{6|A|}$ for all other $a$, we get from \eqref{largeN(a)sum} that
\begin{align*}
    \frac{|D|^2}{3}&\leqslant \sum_{a\in A\setminus{(D+S)}}\left|N(a)\right|\\
    &\leqslant |A|\left|\mathcal{N}\right|+\frac{|D|^2}{6|A|}\left|A\setminus \mathcal{N}\right|\\
    &\leqslant |A|\left|\mathcal{N}\right|+\frac{|D|^2}{6}
\end{align*}
so that $\left|\mathcal{N}\right|\geqslant \frac{|D|^2}{6|A|}$ as desired.
\end{proof}
Next, we show that for at least half the elements $a\in \mathcal{N}$, there are many unordered pairs in $N(a)$ that intersect in a common element. Let us define $$\mathcal{N}(1/3)\!\vcentcolon=\!\left\{a\!\in\! \mathcal{N}: \!\text{ $\exists d(a)\in D$ so that $d(a)\!\in\!P$ for at least $\frac{|N(a)|}{3}$ many $P\!\in \!N(a)$}\right\}.$$
\begin{lemma}
    We have that $\left|\mathcal{N}(1/3)\right|\geqslant \frac{\left|\mathcal{N}\right|}{2}$.
\label{N(1/3)large}
\end{lemma}
\begin{proof}
We again argue by contradiction, so assume that $\left|\mathcal{N}(1/3)\right|< \frac{\left|\mathcal{N}\right|}{2}$. Then by definition, for every $a\in \mathcal{N}\setminus\mathcal{N}(1/3)$ and every $d\in D$, $d$ lies in less than $\frac{|N(a)|}{3}$ of all pairs in $N(a)$. Then pick any $a\in \mathcal{N}\setminus\mathcal{N}(1/3)$ and any pair $P=\{d,d'\}\in N(a)$. The number of pairs $Q\in N(a)$ which intersect $P$ is at most $\frac{2|N(a)|}{3}$ as there are fewer than $\frac{|N(a)|}{3}$ pairs containing $d$, and similarly for $d'$. For this lemma only, we define $T$ to be the set
$$T\vcentcolon=\left\{(a,P,Q):\text{$a\in\mathcal{N}\setminus\mathcal{N}(1/3)$ and $P,Q\in N(a)$ are disjoint}\right\}.$$ We have just shown that if $a\in \mathcal{N}\setminus\mathcal{N}(1/3)$, then for any $P\in N(a)$ there are at least $\frac{|N(a)|}{3}$ distinct $Q\in N(a)$ which are disjoint from $P$, so we get that 
\begin{align*}
    |T|&\geqslant \left|\mathcal{N}\setminus\mathcal{N}(1/3)\right|\left( \min_{a\in\mathcal{N}}|N(a)|\right)\left( \min_{a\in\mathcal{N}}\frac{|N(a)|}{3}\right)\\
    &\geqslant \frac{|D|^2}{12|A|}\left(\frac{|D|^2}{6|A|}\right)\left(\frac{|D|^2}{18|A|}\right)\\
    &> \frac{|D|^6}{2^{11}|A|^3}
\end{align*} using Lemma \ref{Nlarge} to get $\left|\mathcal{N}\setminus\mathcal{N}(1/3)\right|\geqslant \frac{\left|\mathcal{N}\right|}{2}\geqslant  \frac{|D|^2}{12|A|}$, and that by definition of $\mathcal{N}$, $|N(a)|\geqslant \frac{|D|^2}{6|A|}$ for $a\in\mathcal{N}$. Now we can assign a sum to each element of $T$ as follows. For each element $(a,P,Q)\in T$, define $\sigma(a,P,Q)\vcentcolon= \sum_{x\in P}x-\sum_{x\in Q}x$. The point is that $\sigma:T\to G$ takes each value at most $C_1$ times, where $C_1$ is the absolute constant in Lemma \ref{bollobas}.
Indeed, let us assume for a contradiction that $(a_i,P_i,Q_i)\in T$ for $i=1,2,\dots C_1+1$ all have the same image under $\sigma$. Then as $P_i\cap Q_i=\emptyset$ for all $i$ by definition of $T$, Lemma \ref{bollobas} gives two distinct $i,j$ so that $P_i\cap Q_j=\emptyset=P_j\cap Q_i$ and we also have $\sum_{x\in P_i}x-\sum_{x\in Q_i}x=\sigma(a_i,P_i,Q_i)=\sigma(a_j,P_j,Q_j)=\sum_{x\in P_j}x-\sum_{x\in Q_j}x$. This rearranges to $\sum_{x\in P_i\cup Q_j}x=\sum_{x\in P_j\cup Q_i}x$. But $D$ is a dissociated set and as $P_i\cap Q_j=\emptyset=P_j\cap Q_i$, the sets $P_i\cup Q_j$ and $P_j\cup Q_i$ are subsets (and not multisubsets) of $D$ with equal sum so we conclude that $P_i\cup Q_j=P_j\cup Q_i$. By definition of $T$, $P_i$ and $Q_i$ are disjoint and so are $P_j$ and $Q_j$ so we must have that $P_i=P_j$ and $Q_i=Q_j$. Finally, this implies that $a_i=a_j$ (as $P_i\in N(a_i)$ and $P_i=P_j\in N(a_j)$ but by definition \eqref{N(a)definition} each pair $P$ lies in exactly one $N(a)$). So we have a contradiction as we assumed that $(a_i,P_i,Q_i),(a_j,P_j,Q_j)$ were distinct. Hence, we conclude that $\sigma$ takes each value at most $C_1$ times.
\bigskip

On the other hand, if $(a,P,Q)\in T$ then $P,Q\in N(a)$ which means precisely that after writing $P=\{d_1,d_1'\}$ and $Q=\{d_2,d_2'\}$, we have $x(d_i,d'_i)=a$ so that
\begin{align*}
    a+y(d_i,d_i') = (d_i+s_{d_i})+(d_i'+s_{d_i'}),
\end{align*}
for $i=1,2$. Subtracting this equation with $i=2$ from that with $i=1$ shows that
\begin{align*}
    \sigma(a,P,Q) &= d_1+d_1'-d_2-d_2' \\
    &= y(d_1,d_1')-y(d_2,d_2')-s_{d_1}-s_{d_1'}+s_{d_2}+s_{d_2'}\\
    &\in A-A+2S-2S.
\end{align*}
Hence, $\sigma:T\to A-A+2S-2S$ is a map from a set of size $|T|> \frac{|D|^6}{2^{11}|A|^3}$ to a set of size $|A-A+2S-2S|\leqslant |A|^2|S|^4$ which takes each value at most $C_1$ times. We deduce that $\frac{|D|^6}{2^{11}|A|^3}< C_1|A|^2|S|^4$ and rearranging gives that $$|S|>\left(\frac{|D|^6}{2^{11}C_1|A|^5}\right)^{\frac{1}{4}},$$ which is the required contradiction as we assumed \eqref{Sassumptions} and we can take $C>2^{11}C_1$.
\end{proof}
Let us see now what it means that for many $a\in \mathcal{N}$, namely for all $a\in \mathcal{N}(1/3)$, lots of unordered pairs in $N(a)$ contain a common element. So pick $a\in \mathcal{N}(1/3)$, then we can find an element $d(a)\in D$ and distinct pairs $P_1,P_2,\dots,P_m\in N(a)$ with $m\geqslant \frac{|N(a)|}{3}\geqslant \frac{|D|^2}{18|A|}$ (recall that $|N(a)|\geqslant \frac{|D|^2}{6|A|}$ by definition \eqref{mathcalNdefinition} of $\mathcal{N}$) so that each $P_i$ contains $d(a)$. Hence, we can write $P_i=\{d(a),d_i(a)\}$. By definition of $N(a)$, we therefore get the following list of equations
\begin{align}
    a+y(d(a),d_1(a))&= (d(a)+s_{d(a)})+(d_1(a)+s_{d_1(a)})
    \label{translateeq}\\
     a+y(d(a),d_2(a))&= (d(a)+s_{d(a)})+(d_2(a)+s_{d_2(a)})\nonumber\\
     &\,\,\vdots\nonumber\\
    a+y(d(a),d_m(a))&= (d(a)+s_{d(a)})+(d_m(a)+s_{d_m(a)}),\nonumber
\end{align}
with $m\geqslant \frac{|D|^2}{18|A|}$. So for any $a\in \mathcal{N}(1/3)$, we get many equations like this which have a common term on the left hand side, and a common term on the right hand side. We are now almost ready to find a larger set of translates $S'$ so that $\left|(D+S')\cap A\right|\geqslant \left|(D+S)\cap A\right|+\frac{|D|^2}{36|A|}$ and hence finish the proof. There are two cases that we need to consider based on whether many of the elements $y(d(a),d_i(a))$ are in $A\setminus{(D+S)}$ or in $D+S$.
\bigskip

The first case is straightforward now that we have \eqref{translateeq}. Indeed, suppose that for a single $a\in \mathcal{N}(1/3)$, at least half of the elements $y(d(a),d_i(a))$ appearing in \eqref{translateeq} lie in $A\setminus{(D+S)}$. Without loss of generality, we may assume that $y(d(a),d_i(a))\in A\setminus{(D+S)}$ for $i=1,2,\dots,\frac{m}{2}$ with $m\geqslant \frac{|D|^2}{18|A|}$. Then if we set $t=d(a)+s_{d(a)}-a$, the equations \eqref{translateeq} give that
\begin{align*}
    y(d(a),d_i(a)) &= (d(a)+s_{d(a)}-a)+ (d_i(a)+s_{d_i(a)})\\
    &= t+(d_i(a)+s_{d_i(a)})\in D+(S+t)
\end{align*}
for $i=1,2,\dots,\frac{m}{2}$. Then we can take $S'= S\cup(S+t)$ so that $|S'|\leqslant 2|S|$ and $$\left|(D+S')\cap A\right|\geqslant \left|(D+S)\cap A\right|+\frac{m}{2}\geqslant \left|(D+S)\cap A\right|+\frac{|D|^2}{36|A|}$$
since $y(d(a),d_i(a))\in (A\cap(D+S'))\setminus(D+S)$ for $i=1,2,\dots,\frac{m}{2}$. This is the desired conclusion.
\bigskip

In the final case, we may assume that for every $a\in \mathcal{N}(1/3)$, at least half of the elements $y(d(a),d_i(a))$ appearing in \eqref{translateeq} lie in $D+S$. Without loss of generality, we may assume that $y(d(a),d_i(a))\in (D+S)$ for $i=1,2,\dots,\frac{m}{2}$ with $m\geqslant \frac{|D|^2}{18|A|}$. Hence, we can find, for each $a\in \mathcal{N}(1/3)$ and each $1\leqslant i\leqslant\frac{m}{2}$, the elements $e_i(a)\in D$ and $s_i(a)\in S$ so that
\begin{align}
    y(d(a),d_i(a)) = e_i(a)+s_i(a).
\label{y=e+s}
\end{align}
Recall also equation \eqref{translateeq} which says that, for each such $a\in \mathcal{N}(1/3)$ and each $i=1,2,\dots,\frac{m}{2}$, we have
\begin{align}
    a+y(d(a),d_i(a)) = (d(a)+s_{d(a)})+(d_i(a)+s_{d_i(a)}).
\label{translateeq2}
\end{align}
We need one more lemma showing that, under the assumptions of this final case, for every $a\in \mathcal{N}(1/3)$, the element $e_i(a)$ must coincide with $d_i(a)$ for some $i$.
\begin{lemma}
    Assume that for every $a\in \mathcal{N}(1/3)$, we have that $y(d(a),d_i(a))\in (D+S)$ for $i=1,2,\dots,\frac{m}{2}$ and that $m\geqslant \frac{|D|^2}{18|A|}$. Then for every $a\in \mathcal{N}(1/3)$, there exists an $i\in\{1,2,\dots,\frac{m}{2}\}$ so that $e_i(a)=d_i(a)$.
\label{lemmae=d}
\end{lemma}
Assuming this lemma for the moment, we can finish the proof. Rewriting $y(d(a),d_i(a))$ using \eqref{y=e+s} in the equation \eqref{translateeq2} gives
\begin{align}
    a+e_i(a)+s_i(a) = (d(a)+s_{d(a)})+(d_i(a)+s_{d_i(a)})
\label{finish1}
\end{align}
for all $a\in\mathcal{N}(1/3)$ and $i=1,2,\dots,\frac{m}{2}$. By Lemma \ref{lemmae=d}, for each $a\in \mathcal{N}(1/3)$, we can find some $i\leqslant \frac{m}{2}$ so that $e_i(a)=d_i(a)$. Plugging this into \eqref{finish1} and cancelling $d_i(a)=e_i(a)$ on both sides gives
\begin{align*}
    a+s_i(a)= d(a)+s_{d(a)}+s_{d_i(a)}
\end{align*}
so that $a= d(a)+s_{d(a)}+s_{d_i(a)}-s_i(a)\in D+2S-S$ for all $a\in \mathcal{N}(1/3)$. Hence, taking our new set of translates to be $S'=(2S-S)\cup S= 2S-S$, we get that
$$\left|(D+S')\cap A\right|\geqslant \left|(D+S)\cap A\right|+\left|\mathcal{N}(1/3)\right|\geqslant \left|(D+S)\cap A\right|+\frac{|D|^2}{12|A|}$$
as $\left|\mathcal{N}(1/3)\right|\geqslant \frac{\left|\mathcal{N}\right|}{2}\geqslant \frac{|D|^2}{12|A|}$ by Lemmas \ref{Nlarge} and \ref{N(1/3)large} and as $\mathcal{N}\subset A$ is disjoint from $D+S$ by definition \eqref{mathcalNdefinition}. This is the desired conclusion. So we only need to prove Lemma \ref{lemmae=d}.
\begin{proof}[Proof of Lemma \ref{lemmae=d}]
Suppose for a contradiction that the lemma is false. Then there exists some $a\in \mathcal{N}(1/3)$ so that 
\begin{equation}
    e_i(a)\neq d_i(a)
\label{dnoteeq}
\end{equation} for all $i=1,2,\dots,\frac{m}{2}$. Rewriting $y(d(a),d_i(a))$ using \eqref{y=e+s} in the equation \eqref{translateeq2} gives
\begin{align}
    a+e_i(a)+s_i(a) = (d(a)+s_{d(a)})+(d_i(a)+s_{d_i(a)})
\label{finish}
\end{align}
for this supposed counterexample $a\in\mathcal{N}(1/3)$, and every $i=1,2,\dots,\frac{m}{2}$. Hence, if we write $t'=a-(d(a)+s_{d(a)})$, then for each such $i$ we have that
\begin{align*}
    d_i(a)-e_i(a)&= a-(d(a)+s_{d(a)})+s_i(a) -s_{d_i(a)}\\
    &= t'+s_i(a)-s_{d_i(a)}\in t'+S-S.
\end{align*}
However, the set $t'+S-S$ has at most $|S|^2\leqslant|S|^4\leqslant \frac{|D|^6}{C|A|^5} <\frac{|D|^2}{36C_1|A|}\leqslant\frac{m}{2C_1}$ many elements by assumption \eqref{Sassumptions}, as $|D|\leqslant |A|$ since $D\subset A$, and by choosing $C$ sufficiently large in terms of the absolute constant $C_1$. By the pigeonhole principle, out of all $\frac{m}{2}$ possible indices $i$ there exist $C_1+1$ distinct such indices, say $i=1,\dots,C_1+1$, so that 
\begin{equation}
    d_i(a)-e_i(a)=d_{j}(a)-e_j(a)
\label{diffsameoften}
\end{equation} for all $1\leqslant i,j\leqslant C_1+1$. Since $e_i(a)\neq d_i(a)$ by \eqref{dnoteeq}, the sets $P_i=\{e_i(a)\},Q_i=\{d_i(a)\}$ satisfy $P_i\cap Q_i=\emptyset$ so we can apply Lemma \ref{bollobas} to deduce that, without loss of generality, $P_1\cap Q_2=\emptyset=P_2\cap Q_1$. Rearranging \eqref{diffsameoften} and the fact that $D$ is dissociated then yield $\{d_i(a),e_j(a)\}=\{d_j(a),e_i(a)\}$ so by \eqref{dnoteeq} we conclude that $d_i(a)=d_j(a)$. This is the required contradiction (as for a fixed $a$, the elements $d_1(a),\dots,d_m(a)$ that we defined for the equations \eqref{translateeq} are all distinct). This finishes the proof of the lemma.
\end{proof}
This concludes the proof of Proposition \ref{iteratetran}.
\end{proof}
\bibliographystyle{plain}
\bibliography{referencesUniqueSums}
\bigskip

\noindent
{\sc Mathematical Institute, Andrew Wiles Building, University of Oxford, Radcliffe
Observatory Quarter, Woodstock Road, Oxford, OX2 6GG, UK.}\newline
\href{mailto:benjamin.bedert@magd.ox.ac.uk}{\small benjamin.bedert@magd.ox.ac.uk}
\end{document}